\documentclass[11pt]{amsart}
\usepackage{amsmath,amsthm,amsfonts,amssymb,mathrsfs, mathtools}
\usepackage{enumerate}
\usepackage[colorlinks=true, linkcolor=blue, citecolor=magenta, menucolor=black]{hyperref}
\usepackage[left=3.5cm, right=3.5cm]{geometry}
\usepackage[utf8]{inputenc}
\usepackage[demo]{graphicx}
\usepackage[up]{caption}
\usepackage{subcaption}
\usepackage{pgf,tikz}
\usepackage[toc,page]{appendix}
\usetikzlibrary{arrows}
\usetikzlibrary{patterns}
\usepackage{color}

\makeatletter
\@namedef{subjclassname@2020}{%
  \textup{2020} Mathematics Subject Classification}
\makeatother

\usepackage{tikz-cd}

\numberwithin{equation}{section}

\newtheorem{letterthm}{Theorem}

\newtheorem{thm}{Theorem}[section]
\newtheorem{lem}[thm]{Lemma}

\newtheorem{prop}[thm]{Proposition}

\theoremstyle{definition}

\newtheorem{df}[thm]{Definition}

\newtheorem{claim}[thm]{Claim}

\newtheorem*{HS}{Haagerup--St\o rmer's conjecture (HSC)}

\newcommand{\R}{\mathbf{R}}
\newcommand{\C}{\mathbf{C}}

\newcommand{\N}{\mathbf{N}}

\newcommand{\rB}{\operatorname{B}}
\newcommand{\Tr}{\operatorname{Tr}}

\newcommand{\ovt}{\mathbin{\overline{\otimes}}}
\newcommand{\bigovt}{\mathbin{\overline{\bigotimes}}}
\newcommand{\Aut}{\operatorname{Aut}}
\newcommand{\Sd}{\operatorname{Sd}}

\newcommand{\Ad}{\operatorname{Ad}}

\newcommand{\id}{\operatorname{id}}

\newcommand{\Out}{\operatorname{Out}}
\newcommand{\Inn}{\operatorname{Inn}}

\newcommand{\rD}{\operatorname{D}\! }

\newcommand{\rE}{\operatorname{ E}}
\newcommand{\rC}{\operatorname{C}}
\newcommand{\rL}{\operatorname{ L}}
\newcommand{\rd}{\operatorname{ d}}

\newcommand{\Ball}{\operatorname{Ball}}

\allowdisplaybreaks

\begin{document}

\title[Pointwise inner automorphisms of almost periodic factors]{Pointwise inner automorphisms \\ of almost periodic factors}

\begin{abstract}
We prove that a large class of nonamenable almost periodic type ${\rm III_1}$ factors $M$, including all McDuff factors that tensorially absorb $R_\infty$ and all free Araki--Woods factors, satisfy Haagerup--St\o rmer's conjecture (1988): any pointwise inner automorphism of $M$ is the composition of an inner and a modular automorphism. 
\end{abstract}

\author{Cyril Houdayer}
\address{Universit\'e Paris-Saclay \\ Institut Universitaire de France \\  Laboratoire de Math\'ematiques d'Orsay\\ CNRS \\ 91405 Orsay\\ FRANCE}
\email{cyril.houdayer@universite-paris-saclay.fr}
\thanks{CH is supported by Institut Universitaire de France}

\author{Yusuke Isono}
\address{RIMS, Kyoto University, 606-8502 Kyoto, JAPAN}
\email{isono@kurims.kyoto-u.ac.jp}
\thanks{YI is supported by JSPS KAKENHI Grant Number JP20K14324}

\subjclass[2020]{46L10, 46L36, 46L54, 46L55}
\keywords{Almost periodic factors; Pointwise inner automorphisms; Type ${\rm III}$ factors}

\dedicatory{Dedicated to Yasuyuki Kawahigashi on the occasion of his 60th birthday.}
\maketitle

\section{Introduction and statement of the main results}

All the von Neumann algebras we consider in this paper are assumed to have separable predual. We denote by $R_\infty$ the Araki--Woods type ${\rm III}_1$ factor: it is the unique amenable type ${\rm III_1}$ factor with separable predual \cite{Co75, Co85, Ha85}.

Following \cite{HS87}, for any von Neumann algebra $M$, we say that an automorphism $\Theta \in \Aut(M)$ is {\bf pointwise inner} if for any normal state $\psi \in M_*$, there is $u\in \mathcal U(M)$ such that $\Theta(\psi)=u\psi u^*$. All inner automorphisms are pointwise inner. By \cite{Co72}, all modular automorphisms are pointwise inner. In a series of articles \cite{HS87, HS88, HS91}, Haagerup--St\o rmer studied the classification of pointwise inner automorphisms for various classes of von Neumann algebras. For semifinite von Neumann algebras, they showed that all pointwise inner automorphisms are inner \cite{HS87}. For type ${\rm III}_\lambda$ factors with $0 \leq \lambda < 1$, they showed that all pointwise inner automorphisms are compositions of an inner and an extended modular automorphism \cite{HS88}. For type ${\rm III_1}$ factors, they formulated the following conjecture.

\begin{HS}[\cite{HS88}]
Let $M$ be any type ${\rm III}_1 $ factor with separable predual and $\Theta\in \Aut(M)$ any automorphism. The following conditions are equivalent:
\begin{enumerate}[\rm (i)]
	\item $\Theta$ is pointwise inner.
	\item There exist $u\in \mathcal U(M)$ and $t\in \R$ such that $\Theta = \Ad(u)\circ \sigma_t^\varphi$ where $\varphi\in M_*$ is a faithful normal state. 
\end{enumerate}
\end{HS}

Haagerup--St\o rmer solved their conjecture for the Araki--Woods type ${\rm III}_1$ factor $R_\infty$ \cite{HS91}. Up to now, $R_\infty$ was the only type ${\rm III}_1$ factor known to satisfy {\bf HSC}. In this paper, we provide two natural classes of type ${\rm III}_1$ factors satisfying {\bf HSC}.

Following \cite{Co74}, we say that a factor $M$ is {\bf almost periodic} if it possesses an almost periodic faithful normal state $\varphi \in M_\ast$ in the sense that the modular operator $\Delta_\varphi$ is diagonalizable on $\rL^2(M)$. Following \cite{Sh96}, for every countable dense subgroup $\Lambda < \R_{>0}$, denote by $(T_\Lambda, \varphi_\Lambda)$ the unique almost periodic {\bf free Araki--Woods factor} of type ${\rm III_1}$ endowed with its unique free quasi-free state whose Connes'\ $\Sd$ invariant satisfies $\Sd(T_\Lambda) = \Lambda$.

Our first main theorem shows that a large class of almost periodic McDuff type ${\rm III_1}$ factors satisfy {\bf HSC}.

\begin{letterthm}\label{thmA}
	Let $M$ be any almost periodic factor such that $M \cong R_\infty \ovt M$. Then $M$ satisfies {\bf HSC}. 
\end{letterthm}

Combining Theorem \ref{thmA} with the unique McDuff decomposition \cite[Theorem E]{HMV16}, the family $(R_\infty \ovt T_\Lambda)_{\Lambda}$ provides an explicit continuum of pairwise nonisomorphic McDuff type ${\rm III}_1$ factors satisfying {\bf HSC}.

The factors in Theorem \ref{thmA} are all McDuff factors hence non full. In Theorem \ref{thm-HS-free}, we prove that many almost periodic free product factors satisfy {\bf HSC}. Combining Theorem \ref{thm-HS-free} with \cite[Theorem B]{HN18}, our second main theorem shows that a large class of almost periodic full type ${\rm III_1}$ factors satisfy {\bf HSC}.

\begin{letterthm}\label{thmB}
	All almost periodic free Araki--Woods type ${\rm III_1}$ factors $(T_\Lambda)_{\Lambda}$ satisfy {\bf HSC}.
\end{letterthm}

In particular, the family $(T_\Lambda)_{\Lambda}$ provides an explicit continuum of pairwise nonisomorphic full type ${\rm III}_1$ factors satisfying {\bf HSC}.

We introduce the notion of {\bf locally pointwise inner} automorphism with respect to a normal state (see Definition \ref{def-pointinner}). Locally pointwise inner automorphisms share similar properties with pointwise inner automorphisms but they are more flexible. 

Firstly, generalizing the methods developed in \cite{HS91}, we obtain a useful characterization of locally pointwise inner automorphisms with respect to almost periodic states (see Theorem \ref{thm-PI}). Secondly, we prove that in certain situations (e.g.\ tensor products or free products), locally pointwise inner automorphisms remain locally pointwise inner under restriction to subalgebras. This feature is new compared to \cite{HS91}. For this, we exploit the intertwining theory of normal states on von Neumann algebras introduced in \cite{HSV16} (see Theorem \ref{thm-intertwining} and Propositions \ref{prop-stability-tensor} and \ref{prop-stability-free}). The proof of Theorems \ref{thmA} and \ref{thmB} follows by combining our results together with the classification of automorphisms of $R_\infty$ as in \cite{HS91} (see Theorem \ref{thm-HS-conjecture} for a general statement).

{ \hypersetup{linkcolor=black} \tableofcontents}

\section{Preliminaries}

For every von Neumann algebra $M$, we denote by $(M, \rL^2(M), J, \rL^2(M)^+)$ its standard form \cite{Ha73}. For every $\varphi \in (M_\ast)^+$, we denote by $\xi_\varphi \in \rL^2(M)^+$ the unique vector such that $\varphi(x) = \langle x\xi_\varphi, \xi_\varphi \rangle$ for every $x \in M$. For every $\Theta \in \Aut(M)$ and every $\varphi \in M_\ast$, we simply write $\Theta(\varphi) = \varphi \circ \Theta^{-1}$. The predual $M_\ast$ is naturally endowed with the following $M$-$M$-bimodule structure
$$\forall a, b, x \in M, \quad (a\varphi b)(x) = \varphi(b x a).$$

\subsection{Almost periodic factors}

	Let $\Lambda < \R_{>0}$ be any countable subgroup. By regarding $\Lambda$ as a discrete group, we may consider its dual compact group $G=\widehat{\Lambda}$. The inclusion $\beta \colon \Lambda \to \R^*_+$ is continuous so that the dual map 
	$$\widehat{\beta}\colon \R \to G \colon t \mapsto  ( \lambda \mapsto \lambda^{{\rm i} t})$$
is continuous as well. Then $\widehat{\beta}$ has a dense image. When $\Lambda < \R_{>0}$ is cyclic, $G\cong \mathrm{\mathbf T}$ and $\widehat{\beta}$ is surjective. When $\Lambda < \R_{>0}$ is dense, $\widehat{\beta}$ is injective.

Let $M$ be any von Neumann algebra with a faithful normal semifinite weight $\varphi$. Denote by $\sigma_p(\Delta_\varphi)$ the point spectrum of the modular operator $\Delta_\varphi$. For every $\lambda \in \sigma_p(\Delta_\varphi)$, denote by 
 $$M_\varphi(\lambda) = \left \{ x \in M \mid \forall t \in \R, \quad \sigma_t^\varphi(x) = \lambda^{{\rm i} t} x\right \}$$
the subspace of $\lambda$-eigenvectors for $\varphi$. By definition, the centralizer $M_\varphi$ coincides with the subspace $M_\varphi(1)$. We say that $\varphi$ is {\bf almost periodic} \cite{Co74} if the linear span of $M_\varphi(\lambda)$ with $\lambda \in \sigma_p(\Delta_\varphi)$ is $\ast$-strongly dense in $M$. Since $M$ has separable predual, $\sigma_p(\Delta_\varphi)$ is countable. Note that $\varphi$ is semifinite on $M_\varphi$ and $M_\varphi\subset M$ is with faithful normal expectation.
For a countable subgroup $\Lambda < \R_{>0}$, $\varphi$ is called $\Lambda$-{\bf almost periodic} if $\varphi$ is almost periodic and $\sigma_p(\Delta_\varphi)\subset \Lambda$.

We summarize known facts for von Neumann algebras $M$ with almost periodic weights. We refer to \cite{Co74,Dy94} for more details.
\begin{itemize}
	\item A faithful normal semifinite weight $\varphi$ on $M$ is $\Lambda$-almost periodic if and only if the modular action $\sigma^\varphi : \R \curvearrowright M$ extends uniquely continuously to $\sigma^{\Lambda,\varphi}\colon G\curvearrowright  M$ via $\widehat{\beta}$. We fix such $\varphi$.

	\item The crossed product $\rd_{\Lambda,\varphi}(M) =M\rtimes_{\sigma^{\Lambda,\varphi}} G$ is called the {\bf discrete core} with respect to $\Lambda$ and $\varphi$. It admits a canonical trace $\Tr$ satisfying $\Tr\circ \widehat{\sigma}^{\Lambda, \varphi}_\lambda = \lambda \Tr$ for every $\lambda \in \Lambda$, where $\widehat{\sigma}^{\Lambda, \varphi} \colon \Lambda \curvearrowright \rd_{\Lambda,\varphi}(M)$ is the dual action of $\sigma^{\Lambda,\varphi} : G \curvearrowright M$. 

	\item The Takesaki duality shows that 
	$$ \rd_{\Lambda,\varphi}(M)\rtimes_{\widehat{\sigma}^{\Lambda, \varphi}} \Lambda \cong M\ovt \rB(\ell^2(\Lambda)).$$
Since $\Lambda$ is discrete and since the dual action is properly outer (because it is trace-scaling), 
$\rd_{\Lambda, \varphi}(M)\subset M\ovt \rB(\ell^2(\Lambda))$ is a semifinite and irreducible von Neumann subalgebra with faithful normal expectation. In this picture, $\rd_{\Lambda,\varphi}(M)$ coincides with $(M\ovt \rB(\ell^2(\Lambda)))_{\varphi\otimes \Tr(h \, \cdot )}$, where $h \in \rB(\ell^2(\Lambda))$ is the nonsingular operator defined by $h\delta_\lambda = \lambda \delta_\lambda$ for every $\lambda\in \Lambda$. 

	\item For $\lambda\in \Lambda$, let $e_\lambda\in \rB(\ell^2(\Lambda))$ be the minimal projection corresponding to $\delta_\lambda$. By applying the compression map by $1\otimes e_\lambda$, the irreducibility of $\rd_{\Lambda, \varphi}(M)\subset M \ovt \rB(\ell^2(\Lambda))$ implies that $M_\varphi \subset M$ is irreducible. In particular, for any almost periodic weight $\varphi$, we have $M_\varphi'\cap M \subset M_\varphi$. 

	\item Assume that $\sigma_p(\Delta_\varphi)$ generates $\Lambda$. We say that $\varphi$ is {\bf extremal} if $M_\varphi$ is a factor. Then $\varphi$ is extremal if and only if $\rd_{\Lambda,\varphi}(M)$ is a factor. In this case, $\sigma_p(\Delta_\varphi)$ coincides with $\Lambda$. 
\end{itemize}

Recall that any type ${\rm III}$ factor that possesses an extremal almost periodic faithful normal weight is necessarily of type ${\rm III}_\lambda$ with $0 < \lambda \leq 1$ (see \cite[Corollaire 3.2.7]{Co72}). In this case, the centralizer is an irreducible type II subfactor.

\begin{lem}\label{lem-injective-modular}
	Let $M$ be any factor with an extremal almost periodic faithful normal weight $\varphi$. Set $\Lambda = \sigma_p(\Delta_\varphi)$ and $G =\widehat{\Lambda}$. Then for every $g\in G$, 
	$$ \sigma^{\Lambda,\varphi}_g \in \Inn(M) \quad \Leftrightarrow \quad g= 0.$$
\end{lem}
\begin{proof}
	By assumption, the discrete core $\rd_{\Lambda,\varphi}(M)=M\rtimes_{\sigma^{\Lambda,\varphi}} G$ is a factor. Suppose that $\sigma_g^{\Lambda,\varphi}=\Ad(u)$ for some $u\in \mathcal U(M)$. Since  $\sigma_g^{\Lambda,\varphi}(\varphi) = \varphi$, we have $u\in \mathcal U(M_\varphi)$. Since $G$ is commutative, $u^*\lambda_g \in \rL(G)' \cap \rd_{\Lambda, \varphi}(M)$. Since $\Ad(u^*\lambda_g)$ is trivial on $M$, $u^*\lambda_g \in M' \cap \rd_{\Lambda, \varphi}(M)$. Thus, $u^*\lambda_g \in \mathcal Z(\rd_{\Lambda,\varphi}(M)) = \C 1$. It follows that $u\in \C 1$ and $\lambda_g \in \C 1$. This further implies that $\lambda_g = 1$ and so $g=0$.
\end{proof}

\subsection{Intertwining theory of normal positive linear functionals}

Let $M$ be any von Neumann algebra and $\varphi \in (M_\ast)^+$ any normal positive linear functional. Denote by $s(\varphi) \in M$ the support projection of $\varphi$ in $M$. Then $\varphi  = \varphi s(\varphi) = s(\varphi) \varphi$ is a faithful normal positive linear functional on $s(\varphi) M s(\varphi)$. We then simply denote by $M_\varphi \subset s(\varphi) M s(\varphi)$ the centralizer of $\varphi$ on $s(\varphi) M s(\varphi)$. We denote by $\mathcal E_{M, \varphi}$ the nonempty set of all faithful normal positive linear functionals $\Phi \in (M_\ast)^+$ such that $s(\varphi) \in M_{\Phi}$ and $\Phi s(\varphi) =  s(\varphi) \Phi = \varphi$. 

Let $\varphi, \psi \in (M_\ast)^+$, $\Phi \in \mathcal E_{M, \varphi}$ and $\Psi \in \mathcal E_{M, \psi}$. For every $t \in \R$, denote by $w_t = [\rD \Psi : \rD \Phi]_t \in \mathcal U(M)$ Connes' Radon--Nikodym cocycle. Following \cite[Theorem VIII.3.19]{Ta03}, define $[\rD \psi : \rD \Phi]_t = s(\psi) w_t = w_t \sigma_t^{\Phi}(s(\psi)) \in s(\psi) M \sigma_t^\Phi(s(\psi))$ for every $t \in \R$.

Generalizing \cite{HSV16}, we introduce the following terminology regarding the intertwining property of normal positive linear functionals.

\begin{df}
Let $M$ be any von Neumann algebra and $\varphi, \psi \in (M_\ast)^+$ any nonzero positive normal linear functionals. We say that $\psi$ and $\varphi$ are {\bf intertwined in}  $M$ and write $\psi \sim_M \varphi$ if there exist a nonzero partial isometry $v \in s(\psi)Ms(\varphi)$ and $\lambda > 0$ such that $ \psi v = \lambda  v \varphi$. 
\end{df}

If $ \psi v = \lambda  v \varphi$, then we have $v^*v \in M_\varphi$ and $vv^* \in M_\psi$. Since the intertwining property is a symmetric relation, we prefer to use the notation $\psi \sim_M \varphi$ instead of $\psi \preceq_M \varphi$ as in \cite{HSV16}. Observe that if $\psi \sim_M \varphi$, then $u\psi u^* \sim_M \varphi$ for every $u \in \mathcal U(M)$.

The next theorem provides an analytic characterization of the intertwining property. It is a generalization of \cite[Corollary 3.2]{HSV16} to non faithful states (see also \cite[Definition 3.7]{Is19} for a generalization to conditional expectations). 

\begin{thm}\label{thm-intertwining}
Let $M$ be any von Neumann algebra and $\varphi, \psi \in (M_\ast)^+$ any nonzero normal positive linear functionals. Let $\Phi \in \mathcal E_{M, \varphi}$, $\Psi \in \mathcal E_{M, \psi}$ and set $w_t = [\rD \Psi : \rD \Phi]_t \in \mathcal U(M)$ for every $t \in \R$. Let $\mathcal S \subset M$ be any subset that is strongly dense in $M$. The following assertions are equivalent:
\begin{enumerate}[\rm (i)]
\item $\psi \sim_M \varphi$.
\item There exists no net $(t_j)_{j \in J}$ in $\R$ such that we have
$$\forall a, b \in \mathcal S, \quad \lim_{j \to \infty} \varphi(\sigma_{t_j}^{\Phi}(b^*) w_{t_j}^* s(\psi) a) = 0.$$
\item There exists no net $(t_j)_{j \in J}$ in $\R$ such that we have
$$\forall a, b \in M, \quad \lim_{j \to \infty} \varphi(\sigma_{t_j}^{\Phi}(b^*) w_{t_j}^* s(\psi) a) = 0.$$
\end{enumerate}
\end{thm}

\begin{proof}
$(\rm i)  \Rightarrow (\rm iii)$ Let $v \in s(\psi) M s(\varphi)$ be any nonzero partial isometry and $\lambda > 0$ such that $\psi v = \lambda v \varphi$. Then $\psi vv^* = \lambda v \varphi v^* = \lambda  v \Phi v^*$. The chain rule (see \cite[Theorem VIII.3.7]{Ta03}) implies that 
\begin{align*}
\forall t \in \R, \quad vv^* w_t&= [\rD \, (\psi vv^*) : \rD\Phi]_t
	= [\rD \, (\lambda  v \Phi v^*) : \rD\Phi]_t\\
	&=\lambda^{{\rm i} t}  [\rD \, (  v \Phi v^*) : \rD\Phi]_t = \lambda^{{\rm i} t}   v \sigma_t^\Phi( v^*) .
\end{align*}
This implies that for every $t \in \R$, we have $ \lambda^{{\rm i}t} \sigma_t^\Phi(v^*) w_t^* v = v^*v$ and so 
$$|\varphi(\sigma_t^\Phi(v^*) w_t^* s(\psi) v)| = \varphi(v^*v) \neq 0.$$ 

$(\rm ii) \Leftrightarrow (\rm iii)$ Set $\mathcal H = s(\psi) J s(\varphi) J\rL^2(M)$ and observe that $s(\psi)\mathcal Ss(\varphi) \xi_\varphi$ is dense in $\mathcal H$. We may then define the strongly continuous unitary representation $\rho : \R \to \mathcal U(\mathcal H)$ by the formula $\rho_t(s(\psi)a \xi_\varphi) = s(\psi) w_t \sigma_t^\Phi(a) \xi_\varphi$ for every $a \in \mathcal S$. Then the assertions $(\rm ii)$ and $(\rm iii)$ are both equivalent to the fact that $\rho : \R \to \mathcal U(\mathcal H)$ is not weakly mixing.

 $(\rm iii) \Rightarrow (\rm i)$ Since $\rho : \R \to \mathcal U(\mathcal H)$ is not weakly mixing and since $\R$ is abelian, there exist $s \in \R$ and a unit vector $\eta \in \mathcal H$ such that $\rho_t(\eta) = \exp({\rm i}st) \, \eta$ for every $t \in \R$. 
Set $\mathcal C\coloneqq \Ball(s(\psi)Ms(\varphi))\xi_\varphi \subset \mathcal H$ and observe that it is a closed convex subset (because $\Ball(s(\psi)Ms(\varphi))$ is $\sigma$-weakly compact). Take the unique element $x \in \Ball(s(\psi)Ms(\varphi))$ such that $x\xi_\varphi$ is the orthogonal projection of $\eta$ onto $\mathcal C$. Since $s(\psi)Ms(\varphi)\xi_\varphi$ is dense in $\mathcal H$ and $\eta \neq 0$, we necessarily have $x \neq 0$. For every $t \in \R$, we have that $\rho_t(x\xi_\varphi)$ is the orthogonal projection of $\rho_t(\eta) = \exp({\rm i}st) \, \eta$ onto the closed convex subset $\rho_t(\mathcal C) = \mathcal C$ and hence $\rho_t(x\xi_\varphi) = \exp({\rm i}st) \, x\xi_\varphi$. Therefore, we have $s(\psi)w_t \sigma_t^\Phi(x) = \exp({\rm i} st) \, x$ for every $t \in \R$. Write $x = v |x|$ for the polar decomposition of $x \in M$. Then $v \in s(\psi)Ms(\varphi)$ is a nonzero partial isometry such that $s(\psi)w_t \sigma_t^\Phi(v) = \exp({\rm i} st) \, v$ for every $t \in \R$. Observe that $p = v^*v \in M_\varphi$ and $q = vv^* \in M_\psi$.

Define the normal positive linear functional $\rho \in (M_\ast)^+$ by the formula $\rho = \exp(s) \, v\Phi v^*$. Observe that the support of $\rho$ equals $q$. Then for every $t \in \R$, we have $[\rD \rho : \rD \Phi]_t = \exp({\rm i}st) \, v \sigma_t^\Phi(v^*) \in q M \sigma_t^\Phi(q)$. It holds that 
\begin{align*}
\forall t \in \R, \quad [\rD \rho : \rD \Psi]_t &= [\rD \rho : \rD\Phi]_t \, w_t^* \\
&= \exp({\rm i}st) \, v \sigma_t^\Phi(v^*)  w_t^* \\
&= s(\psi)w_t \sigma_t^\Phi (q) w_t^* \\
&= s(\psi)\sigma_t^\Psi(q) =q\\
&= [\rD \Psi q : \rD \Psi]_t.
\end{align*}
Then \cite[Th\'eor\`eme 1.2.4]{Co72} (or more generally \cite[Theorem VIII.3.19]{Ta03}) implies that $\rho = \Psi q = \psi q$. Letting $\lambda = \exp(s)$, we have $\psi q = \rho =  \lambda v \Phi v^*$ and so $\psi v = \lambda v \Phi = \lambda v \varphi$.
\end{proof}

The following lemma provides a sufficient condition for transitivity of the symmetric relation $\sim_M$.

\begin{lem}\label{lem-transitivity}
Let $M$ be any von Neumann algebra and $\varphi \in M_\ast$ any faithful normal state such that $M_\varphi$ is a factor. For all $\psi_1, \psi_2 \in (M_\ast)^+$, if $\psi_1 \sim_M \varphi$ and $\psi_2 \sim_M \varphi$, then $\psi_1 \sim_M \psi_2$.
\end{lem}

\begin{proof}
Let $\psi_1, \psi_2 \in (M_\ast)^+$. Assume that $\psi_1 \sim_M \varphi$ and $\psi_2 \sim_M \varphi$. Then there exist nonzero partial isometries $v_1 \in s(\psi_1)M$ and $v_2 \in s(\psi_2)M$ and $\lambda_1, \lambda_2 > 0$ such that $\psi_1 v_1 = \lambda_1 v_1 \varphi$ and $\psi_2 v_2 = \lambda_2 v_2 \varphi$. Since $v_1^*v_1, v_2^*v_2 \in M_\varphi$ and since $M_\varphi$ is a factor, there exists a partial isometry $u \in M_\varphi$ such that $w = v_1 u v_2^*\in s(\psi_1) M s(\psi_2)$ is a nonzero partial isometry. It follows that
$$\psi_1 w = \psi_1 v_1 u v_2^* = \lambda_1 v_1 \varphi u v_2^* = \lambda_1 v_1 u \varphi v_2^* = \lambda_1\lambda_2^{-1} v_1 u v_2^* \psi_2 = \lambda_1 \lambda_2^{-1} w \psi_2.$$
This implies that $\psi_1 \sim_M \psi_2$.
\end{proof}

\section{Locally pointwise inner automorphisms}

\subsection{Locally pointwise inner automorphisms for weights}
Recall that a faithful normal semifinite weight $\varphi$ on a von Neumann algebra $M$ is {\bf strictly semifinite} if $\varphi|_{M_\varphi}$ is still semifinite. Inspired by \cite{HS91}, we introduce the following terminology.

\begin{df}\label{def-pointinner}
	Let $M$ be any von Neumann algebra and $\varphi$ any faithful normal strictly semifinite weight on $M$. Let $\Theta\in \Aut(M)$ be any automorphism. We say that $\Theta\in \Aut(M)$ is 
	 {\bf locally pointwise inner for} $\varphi$ if $\Theta(\varphi)=\varphi$ and for any nonzero projection $p\in M_\varphi$ with $\varphi(p)<\infty$ and any positive invertible element $h\in p M_\varphi p$, we have $\Theta(\varphi_h) \sim_M \varphi_h$.
\end{df}

We note that all inner automorphisms $\Ad(u)$ with $u \in \mathcal U(M_\varphi)$ are locally pointwise inner for $\varphi$. Also if $\Theta$ is locally pointwise inner for $\varphi$, then so are $\Theta^{-1}$ and $\Ad(u)\circ \Theta$ for any $u\in \mathcal U(M_\varphi)$. 

\begin{lem}\label{lem-pointinner}
	Let $M$ be any von Neumann algebra and $\varphi$ any faithful normal strictly semifinite weight on $M$. Let $\Theta\in \Aut(M)$ be any pointwise inner automorphism.
	 Then there exists $u \in \mathcal U(M)$ such that $\Ad(u) \circ \Theta$ is locally pointwise inner for $\varphi$.
\end{lem}
\begin{proof}
Let $\Theta \in \Aut(M)$ be any pointwise inner automorphism. Then there exists $u \in \mathcal U(M)$ such that $(\Ad(u) \circ \Theta)(\varphi) = \varphi$ (see \cite[Lemma 2.1]{HS88}). Set $\Theta^u = \Ad(u) \circ \Theta$. Let  $p\in M_\varphi$ be any nonzero projection with $\varphi(p)<\infty$ and $h\in p M_\varphi p$ any positive invertible element. Since $\Theta^u$ is pointwise inner, there is a unitary $v\in \mathcal U(M)$ such that $\Theta^u(\varphi_h) = v\varphi_h v^*$. By comparing support projections, we have $\Theta^u(p) = vpv^*$. Then letting $w =vp =\Theta^u(p)v$, we obtain $\Theta^u(\varphi_h)w = w\varphi_h $. Thus, $\Theta^u$ is locally pointwise inner for $\varphi$.
\end{proof}

Next, we observe that locally pointwise inner automorphisms behave well with respect to amplifications. Denote by $\Tr$ the unique faithful normal semifinite trace on $\rB(\ell^2)$ such that all minimal projections $e \in \rB(\ell^2)$ satisfy $\Tr(e) = 1$.

\begin{lem}\label{lem-amplification}
Let $M$ be any von Neumann algebra and $\varphi \in M_\ast$ any faithful normal state such that $M_\varphi$ is a type ${\rm II_1}$ factor. Let $\Theta \in \Aut(M)$ be any automorphism and consider $\Theta \otimes \id \in \Aut(M \ovt \rB(\ell^2))$. 

If $\Theta$ is locally pointwise inner for $\varphi$, then $\Theta \otimes \id$ is locally pointwise inner for $\varphi \otimes \Tr$.
\end{lem}

\begin{proof}
Assume that $\Theta \in \Aut(M)$ is locally pointwise inner for $\varphi$. Set $\mathcal M = M \ovt \rB(\ell^2)$ and $\Phi = \varphi \otimes \Tr$. Since $\Theta(\varphi) = \varphi$, we have $(\Theta \otimes \id)(\Phi) = \Phi$. Let $p \in \mathcal M_\Phi$ be any nonzero projection such that $\Phi(p) < \infty$ and $h \in \mathcal M_\Phi$ any positive invertible element. Choose a maximal abelian subalgebra $A \subset p \mathcal M_\Phi p$ such that $h \in A$. Since $p\mathcal M_\Phi p$ is a type ${\rm II_1}$ factor, $A \subset p \mathcal M_\Phi p$ is diffuse and so we may choose a nonzero projection $p_0 \in A$ such that $\Phi(p_0) \leq 1$. Note that $p_0 h = hp_0$ is a positive invertible element in $p_0 \mathcal M_\Phi p_0$. Since $\mathcal M_\Phi = M_\varphi \ovt \rB(\ell^2)$ is a type ${\rm II}_\infty$ factor, there exists $u \in \mathcal U(\mathcal M_\Phi)$, a projection $q_0 \in M_\varphi$ and a minimal projection $e_0 \in \rB(\ell^2)$ such that $u p_0 u^* = q_0 \otimes e_0$. Denote by $h_0 \in q_0M_\varphi q_0$ the positive invertible element such that $ u(hp_0)u^* = h_0 \otimes e_0$. Since $\Theta \in \Aut(M)$ is locally pointwise inner for $\varphi$, there exist a nonzero partial isometry $v \in \Theta(q_0) M q_0$ and $\lambda > 0$ such that $\Theta(\varphi_{h_0}) v = \lambda v \varphi_{h_0}$. Then $(\Theta \otimes \id)(\Phi_{h_0 \otimes e_0}) (v \otimes e_0) = \lambda (v \otimes e_0) \Phi_{h_0 \otimes e_0}$.

Observe that $\Phi_{h_0 \otimes e_0} =  u \Phi_{hp_0} u^*$. Then we obtain 
\begin{align*}
(\Theta \otimes \id)( u) ( \Theta \otimes \id) (\Phi_{hp_0}) (\Theta \otimes \id)(u^*) (v \otimes e_0) &= (\Theta \otimes \id)(\Phi_{h_0 \otimes e_0}) (v \otimes e_0) \\
&=  \lambda (v \otimes e_0) \Phi_{h_0 \otimes e_0} \\
&= \lambda (v \otimes e_0) u \Phi_{hp_0} u^*
\end{align*}
Letting $w = (\Theta \otimes \id)(u^*) (v \otimes e_0) u \in (\Theta \otimes \id)(p_0)\mathcal M p_0$, we have $w \neq 0$ and
$$ ( \Theta \otimes \id) (\Phi_h) w = ( \Theta \otimes \id) (\Phi_{hp_0}) w  = \lambda w \Phi_{p_0h} = \lambda w \Phi_{h}.$$
Therefore, $\Theta \otimes \id$ is locally pointwise inner for $\Phi = \varphi \otimes \Tr$.
\end{proof}

\subsection{A characterization of locally pointwise inner automorphisms for extremal almost periodic states}
Before proving the main characterization, we need some preliminary results on almost periodic weights.

\begin{lem}\label{lem-eigenvector-unitary}
Let $(M, \varphi)$ be any type $\rm III$ factor endowed with an extremal almost periodic faithful normal semifinite weight such that $\varphi(1) = \infty$. 
\begin{enumerate}[\rm (i)]

	\item Let $v\in M$ be any partial isometry that is an eigenvector for $\varphi$. If $v^*v$ or $vv^*$ is finite in $M_\varphi$, then there is a unitary $V \in \mathcal U(M)$ that is an eigenvector for $\varphi$ and such that $Vv^*v=vv^*V = v$.

	\item Let $\mathcal G\subset \Aut(M_\varphi)$ be the set of all automorphisms of the form $\Ad(u)$ where $u\in \mathcal U(M)$ is an eigenvector for $\varphi$. Then the image of $\mathcal G$ in $\Out(M_\varphi)$ is countable.

\end{enumerate}
\end{lem}
\begin{proof}
	$(\rm i)$ Let $\lambda \in \sigma_p(\Delta_\varphi)$ and $v \in M$ a partial isometry that is a $\lambda$-eigenvector for $\varphi$. Set $p=vv^*$ and $q=v^*v$. Since $\Ad(v) : qM_\varphi q \to pM_\varphi p$ is an isomorphism, it follows that $p$ is finite in $M_\varphi$ if and only if $q$ is finite in $M_\varphi$. By the assumption, both $p$ and $q$ are finite in $M_\varphi$. Since $M_\varphi$ is a type ${\rm II}_\infty$ factor, there are two families of mutually orthogonal equivalent projections $(p_n)_{n\in \N}$ with $p_0=p$ and $(q_n)_{n\in \N}$ with $q_0=q$ in $M_\varphi$. Let $(v_n)_{n \in \N}, (w_n)_{n \in \N}$ be families of partial isometries in $M_\varphi$ such that 
	$$\forall n \in \N, \quad  v_nv_n^*=p,\quad v_n^*v_n=p_n,\quad  w_nw_n^*=q,\quad w_n^*w_n=q_n.$$
Then $V = \sum_{n \in \N} v_n^* v w_n \in \mathcal U(M)$ is a $\lambda$-eigenvector for $\varphi$ such that $Vq=pV = v$.

	$(\rm ii)$ Let $\lambda \in \sigma_p(\Delta_\varphi)$. If $u,w\in \mathcal U(M)$ are $\lambda$-eigenvectors for $\varphi$, then $w^*u\in \mathcal U(M_\varphi)$. This shows that all the automorphisms $\Ad(u)|_{M_\varphi}$ where $u \in \mathcal U(M)$ is a $\lambda$-eigenvector are unitarily equivalent in $\Aut(M_\varphi)$ and so determine the same element in $\Out(M_\varphi)$. Since $\sigma_p(\Delta_\varphi)$ is countable,  the image of $\mathcal G$ in $\Out(M_\varphi)$ is countable.
\end{proof}

The next proposition generalizes \cite[Theorem 2.1]{HS91} to the case of locally pointwise inner automorphisms. 

\begin{prop}\label{HS-prop}
	Let $(M, \varphi)$ be any type $\rm III$ factor endowed with an extremal almost periodic faithful normal semifinite weight such that $\varphi(1) = \infty$. Let $\Theta\in \Aut(M)$ be any automorphism that is locally pointwise inner for $\varphi$. Then $\Theta$ is inner on $M_\varphi$.
\end{prop}
\begin{proof}
Set $\Lambda = \sigma_p(\Delta_\varphi)$ and $G = \widehat \Lambda$. Since $\Theta \in \Aut(M)$ is locally pointwise inner for $\varphi$, we have $\Theta(\varphi) = \varphi$ and $\Theta \circ \sigma_g^{\Lambda, \varphi} = \sigma_g^{\Lambda, \varphi} \circ \Theta$ for every $g \in G$. We denote by $\mathcal{G}\leq \Aut(M_\varphi)$ the subgroup generated by $\Theta|_{M_\varphi}$,  $\Inn(M_\varphi)$, and all $\Ad(u)$'s where $u\in \mathcal U(M)$ is a $\lambda$-eigenvector for $\varphi$ and $\lambda \in \Lambda$ is arbitrary. Then by Lemma \ref{lem-eigenvector-unitary}, the image of $\mathcal{G}$ in $\Out(M_\varphi)$ is countable. 

By \cite[Theorem 4.2]{Po83}, there exists a masa $A\subset M_\varphi$ generated by finite projections in $M_\varphi$ that is $\mathcal{G}$-singular in the following sense:
\begin{itemize}
	\item If $\alpha\in \mathcal{G}$ is such that $\alpha(Ae)\subset A$ for some projection $e\in A$, then $\alpha|_{eM_{\varphi}e} = \Ad(u)$ for some $u\in \mathcal U(Ae)$.
\end{itemize}
Note that $\varphi|_A$ is semifinite. We fix a positive invertible element $h\in A$ that generates $A \subset M_\varphi$. Let $p\in A$ be any nonzero projection such that $\varphi(p)<\infty$. Since $\Theta $ is locally pointwise inner for $\varphi$, there exists a nonzero element $x\in\Theta(p) Mp$ and $\kappa >0$ such that $\Theta(\varphi_{hp})x= \kappa x\varphi_{hp}$. 
\begin{claim}
	There exist $\lambda\in \Lambda$ and a nonzero partial isometry $v_\lambda \in M$ that is a $\lambda$-eigenvector for $\varphi$ such that 
$$v_\lambda^*v_\lambda \in Ap, \quad v_\lambda v_\lambda^*\in \Theta(Ap) \quad \text{and} \quad \Theta(hp) v_\lambda =  \kappa\lambda^{-1 }v_\lambda hp.$$
\end{claim}
\begin{proof}
	For every $\lambda\in \Lambda$, the projection $\rE_\lambda : M \to M_\varphi(\lambda)$ onto the subspace of $\lambda$-eigenvectors for $\varphi$ is given by
	$$\forall a\in M, \quad \rE_\lambda (a) = \int_G \overline{\langle g,\lambda \rangle} \sigma^{\Lambda,\varphi}_g(a) \rd \!g.$$
Let $\lambda\in \Lambda $ be such that $x_\lambda=\rE_\lambda(x)\neq 0$. Note that $x_\lambda \in \Theta(p)M p$. Since for all $g\in G$, we have $\Theta(\varphi_{hp})\sigma^{\Lambda,\varphi}_g(x)=\kappa\sigma^{\Lambda,\varphi}_g(x)  \varphi_{hp}$, we infer that $\Theta(\varphi_{hp})x_\lambda=\kappa x_\lambda \varphi_{hp}$. Since $x_\lambda$ is a $\lambda$-eigenvector for $\varphi$, this implies that $ \lambda \Theta(hp) x_\lambda =  \kappa x_\lambda hp$. Write $x_\lambda = v_\lambda |x_\lambda|$ for the polar decomposition of $x_\lambda \in M$. Then it is easy to see that $\lambda \Theta(hp) v_\lambda = \kappa v_\lambda hp$ where
	$$ v_\lambda^*v_\lambda \in \{hp\}' \cap pM_\varphi p=Ap \quad \text{and} \quad v_\lambda v_\lambda^*\in \{\Theta(hp)\}'\cap \Theta(pMp) = \Theta(Ap).$$
This finishes the proof of the claim.
\end{proof}

Set $q=v_\lambda^*v_\lambda\in Ap$ and $\Theta(r)=v_\lambda v_\lambda^*\in \Theta(Ap)$. Since $M_\varphi$ is a type ${\rm II}_\infty$ factor, by Lemma \ref{lem-eigenvector-unitary}, there is a unitary $v\in \mathcal U(M)$ that is a $\lambda$-eigenvector for $\varphi$ such that $v q = \Theta(r) v =v_\lambda$. We have
	$$v^* \Theta(hr) v=  \kappa\lambda^{-1}hq  \quad \text{and} \quad v^*\Theta(r) v = q.$$
Since $h$ generates $A$ and since $q,r\in A$, we have $(\Ad(v^*)\circ \Theta)|_{Ar}=Aq\subset A$. 
Since $A$ is $\mathcal G$-singular, we can write $(\Ad(v^*)\circ \Theta) |_{rM_\varphi r} = \Ad(u)|_{rM_\varphi r}$ for some $u\in \mathcal U(Ar)$. This implies that $(\Ad(v^*)\circ \Theta)(r) = r$ and so $q=r$. Observe that 
	$$\varphi(r) = \varphi(\Theta(r))=\varphi(v_\lambda v_\lambda^*)=\lambda \varphi(v_\lambda^* v_\lambda)=\lambda \varphi(q) = \lambda \varphi(r).$$
This implies $\lambda=1$. Then $v\in \mathcal U(M_\varphi)$ and the equation $(\Ad(v^*)\circ \Theta) |_{rM_\varphi r} = \Ad(u)|_{rM_\varphi r}$ shows that $\Theta|_{M_\varphi}$ is not properly outer. Since $M_\varphi$ is a factor, it follows that $\Theta|_{M_\varphi}$ is inner.
\end{proof}

Let $(M, \varphi)$ be any von Neumann algebra endowed with an almost periodic faithful normal state. Denote by $\Lambda < \R^*_+$ the group generated by the point spectrum $\sigma_p(\Delta_\varphi)$ and set $G=\widehat{\Lambda}$. Denote by $\rd_{\Lambda,\varphi}(M)=M\rtimes_{\sigma^{\Lambda,\varphi}} G$ the discrete core of $M$ with respect to $\Lambda$ and $ \varphi$. 

Let $\Theta \in \Aut(M)$ be any automorphism such that $\Theta(\varphi) = \varphi$. Since $\Theta$ commutes with $\sigma^{\Lambda,\varphi}$, it has a unique extension $\widetilde{\Theta}\in \Aut(\rd_{\Lambda,\varphi}(M))$ defined by
	$$\forall g \in G, \quad \widetilde{\Theta}(\lambda_g)=\lambda_g \quad \text{and}\quad   \widetilde{\Theta}|_M=\Theta.$$
Note that $\widetilde{\Theta}$ commutes with the dual action $\widehat \sigma^{\Lambda, \varphi} : \Lambda \curvearrowright \rd_{\Lambda, \varphi}(M)$.

We obtain the following characterization of locally pointwise inner automorphisms for extremal almost periodic states.

\begin{thm}\label{thm-PI}
	Let $(M, \varphi)$ be any factor endowed with an extremal almost periodic faithful normal state. Set $\Lambda =\sigma_p(\Delta_\varphi)$ and $G=\widehat{\Lambda}$. Let $\Theta\in \Aut(M)$ be any automorphism such that $\Theta(\varphi)=\varphi$. The following assertions are equivalent:
\begin{enumerate}[\rm (i)]
	\item $\Theta$ is locally pointwise inner for $\varphi$.
	
	\item $\Theta$ is inner on $M_\varphi$.
	
	\item $\widetilde \Theta$ is inner on $\rd_{\Lambda,\varphi}(M)$.

	\item There exist $u\in \mathcal U(M_\varphi)$ and $g\in G$ such that $\Theta = \Ad(u)\circ \sigma_g^{\Lambda, \varphi}$.
\end{enumerate}
\end{thm}
\begin{proof}
	$(\rm i) \Rightarrow (\rm ii)$ By Lemma \ref{lem-amplification}, $\Theta \otimes \id \in \Aut(M \ovt \rB(\ell^2))$ is locally pointwise inner for $\varphi \otimes \Tr$. By Proposition \ref{HS-prop}, $\Theta \otimes \id$ is inner on $(M \ovt \rB(\ell^2))_{\varphi \otimes \Tr} = M_\varphi \ovt \rB(\ell^2)$. It follows that $\Theta$ is inner on $M_\varphi$.
	
	$(\rm ii) \Rightarrow (\rm iii)$ As we have explained in the preliminaries, there is an embedding
	$$ M_\varphi \otimes \C1 \subset   (M\ovt \rB(\ell^2(\Lambda)))_{\varphi\otimes \Tr(h\, \cdot \, )}=\rd_{\Lambda, \varphi}(M)$$
and any projection $p=1\otimes e_\lambda$ for $\lambda\in \Lambda$ satisfies $p \rd_{\Lambda, \varphi}(M) p = M_\varphi\otimes \C e_\lambda \cong M_\varphi$.
Observe that $\widetilde{\Theta}(p)=p$ and that $\widetilde{\Theta} |_{p \rd_{\Lambda, \varphi}(M) p} $ coincides with $\Theta|_{M_\varphi}$ via the above isomorphism. 
Then since $\Theta|_{M_\varphi}$ is inner, $\widetilde{\Theta}|_{p \rd_{\Lambda, \varphi}(M)p}$ is inner as well. This shows that $\widetilde{\Theta}$ is not properly outer. Since $\rd_{\Lambda, \varphi}(M)$ is a factor, $\widetilde{\Theta}$ is inner on $\rd_{\Lambda, \varphi}(M)$. 
	
	$(\rm iii) \Rightarrow (\rm iv)$ There exists $u\in \mathcal U(\rd_{\Lambda, \varphi}(M))$ such that $\widetilde{\Theta} = \Ad(u)$. Since $\widetilde{\Theta}$ commutes with the dual action $\widehat \sigma^{\Lambda, \varphi}$, for every $\lambda\in \Lambda$ and every $x\in \rd_{\Lambda, \varphi}(M)$, we have
	$$ \widehat{\sigma}^{\Lambda,\varphi}_\lambda(uxu^*)= \widehat{\sigma}^{\Lambda,\varphi}_\lambda\circ \widetilde{\Theta}(x)= \widetilde{\Theta}\circ \widehat{\sigma}^{\Lambda,\varphi}_\lambda(x) = u \widehat{\sigma}^{\Lambda,\varphi}_\lambda(x)u^*$$
which further implies 
	$$ v_\lambda=u^*\widehat{\sigma}^{\Lambda,\varphi}_\lambda(u) \in \rd_{\Lambda, \varphi}(M)'\cap \rd_{\Lambda, \varphi}(M)=\C 1.$$
Observe that $\Lambda \to \C : \lambda \mapsto v_\lambda$ is a character (indeed it is a cocycle for $\widehat{\sigma}^{\Lambda,\varphi}$ but $\widehat{\sigma}^{\Lambda,\varphi}$ is trivial on $\C$). Thus there is $g\in \widehat{\Lambda} = G$ such that $\langle g,\lambda\rangle = v_\lambda$ for every $\lambda \in \Lambda$. This means
	$$\forall \lambda\in \Lambda, \quad \widehat{\sigma}^{\Lambda,\varphi}_\lambda(u)=\langle g,\lambda\rangle u.$$
Observe that the group element $\lambda_g\in \rd_{\Lambda, \varphi}(M)$ satisfies the same condition as $u$, hence $\lambda_g u^*$ is invariant under the dual action, which implies $\lambda_g u^*\in M$. We get $u=v\lambda_g$ for some $v\in M$, hence 
	$$\Theta =\Ad(u)|_M= \Ad(v\lambda_g )|_M = \Ad(v)\circ \widehat{\sigma}^{\Lambda,\varphi}_g.$$
Since $\varphi = \Theta(\varphi) = v\varphi v^*$, we moreover have $v \in \mathcal U(M_\varphi)$. 

	$(\rm iv) \Rightarrow (\rm i)$ Let $p \in M_\varphi$ be any nonzero projection and $h\in p M_\varphi p$ any positive invertible element. Since $\sigma^{\Lambda,\varphi}(h)=h$, we have $\sigma^{\Lambda,\varphi} (\varphi_h) = \varphi_h$ and so
	$$ \Theta(\varphi_h) = (\Ad(u) \circ \sigma_g^{\Lambda, \varphi}) (\varphi_h) =\Ad(u)(\varphi_h) = u \varphi_h u^*.$$
By comparing support projections, we have $\Theta(p) = upu^*$. Then letting $v =up =\Theta(p)u$, we obtain $\Theta(\varphi_h)v = v\varphi_h $. Thus, $\Theta$ is locally pointwise inner for $\varphi$.
\end{proof}

\section{Intertwining stable inclusions}

In this section, we investigate intertwining stable inclusions of von Neumann algebras. We introduce the following terminology.

\begin{df}\label{df-stability}
Let $N \subset M$ be any inclusion of von Neumann algebras with faithful normal conditional expectation $\rE : M \to N$. Let $\varphi \in M_\ast$ be any faithful normal state such that $\varphi \circ \rE = \varphi$. We say that $(M, N, \varphi)$ is {\bf intertwining stable} if the following condition is satisfied: 
\begin{itemize}
\item For every $\Theta \in \Aut(M)$ such that $\rE \circ \Theta = \Theta \circ \rE$ and $\Theta(\varphi) = \varphi$ and for every $\psi \in (N_\ast)^+$, if $\Theta (\psi \circ \rE) \sim_M \psi \circ \rE$, then $\Theta_N(\psi) \sim_N \psi$ where $\Theta_N = \Theta|_N$.
\end{itemize}
\end{df}

The notion of intertwining stable inclusion behaves well with respect to locally pointwise inner automorphisms.

\begin{prop}\label{prop-stable}
Let $N \subset M$ be any inclusion of von Neumann algebras with faithful normal conditional expectation $\rE : M \to N$. Let $\varphi \in M_\ast$ be any faithful normal state such that $\varphi \circ \rE = \varphi$. Assume that $(M, N, \varphi)$ is intertwining stable. Let $\Theta \in \Aut(M)$ be any automorphism such that $\rE \circ \Theta = \Theta \circ \rE$ and $\Theta(\varphi) = u\varphi u^*$ for some $u \in \mathcal U(N)$. Set $\Theta_N = \Theta|_N$. Let $\psi$ be any strictly semifinite weight on $N$.

If $\Theta$ is locally pointwise inner for $\psi \circ \rE$, then $\Theta_N$ is locally pointwise inner for $\psi$.
\end{prop}

\begin{proof}
Assume that $\Theta$ is locally pointwise inner for $\psi \circ \rE$. To show that $\Theta_N$ is locally pointwise inner for $\psi$, take any nonzero projection $p \in N_\psi$ such that $\psi(p) < \infty$ and any positive invertible element $h \in p N_\psi p$. Then since $N_\psi \subset M_{\psi\circ E}$ and $(\psi \circ \rE)_h = \psi_h \circ \rE$, the assumption on $\Theta $ implies  $\Theta(\psi_h \circ \rE) \sim_M \psi_h \circ \rE$. Set $\Theta^u = \Ad(u^*) \circ \Theta$ and $\Theta_N^u = \Theta^u|_N$. Then we have $\Theta^u(\psi_h \circ \rE) \sim_M \psi_h \circ \rE$. Since $(M, N, \varphi)$ is intertwining stable and since $\Theta^u(\varphi) = \varphi$, we have $\Theta^u_N(\psi_h) \sim_N \psi_h$. This further implies that $\Theta_N(\psi_h) \sim_N \psi_h$. Thus, $\Theta_N$ is locally pointwise inner for $\psi$.
\end{proof}

We observe that inclusions arising from tensor product von Neumann algebras are intertwining stable.

\begin{prop}\label{prop-stability-tensor}
Let $M=N\ovt P$ be any tensor product von Neumann algebra. Let $\varphi_N \in N_\ast$ and $\varphi_P \in P_\ast$ be any faithful normal states and write $\varphi_M = \varphi_N \otimes \varphi_P \in M_\ast$. Then $(M, N, \varphi_M)$ is intertwining stable.
\end{prop}

\begin{proof}
Denote by $\rE=\id_N\otimes \varphi_P : M \to N$ the faithful normal conditional expectation associated with $\varphi_P \in P_\ast$. The fact that $(M, N, \varphi_M)$ is intertwining stable easily follows from the following general result.

\begin{claim}
For all $\psi_1, \psi_2 \in (N_\ast)^+$, if $\psi_1 \not\sim_N \psi_2$, then $\psi_1 \circ \rE \not\sim_M \psi_2 \circ \rE$.
\end{claim}

Let $\psi_1, \psi_2 \in (N_\ast)^+$ and assume that $\psi_1 \not\sim_N \psi_2$. For every $i \in \{1, 2\}$, $\psi_i \circ \rE = \psi_i \otimes \varphi_P$. For every $i \in \{1, 2\}$, choose $\Psi_i \in \mathcal E_{N, \psi_i}$ and note that $\Psi_i \otimes \varphi_P = \Psi_i \circ \rE \in \mathcal E_{M, \psi_i \circ \rE}$. For every $t \in \R$, set $w_t = [\rD \Psi_2 : \rD \Psi_1]_t \in \mathcal U(N)$. By Theorem \ref{thm-intertwining}, there exists a net $(t_j)_{j \in J}$ in $\R$ such that
\begin{equation}\label{eq-convergence1}
\forall a, b \in N, \quad \lim_{j \to \infty} \psi_1(\sigma_{t_j}^{\Psi_1}(b^*) w_{t_j}^* s(\psi_2)a) = 0.
\end{equation}
Let $a, b \in M$ and $x, y \in P$. Then \eqref{eq-convergence1} implies that
\begin{align*}
(\psi_1 \otimes \varphi_P)(\sigma_{t_j}^{\Psi_1 \otimes \varphi_P}(b^* \otimes y^*)(w_{t_j}^*s(\psi_2) \otimes 1)(a \otimes x)) &=  \psi_1(\sigma_{t_j}^{\Psi_1}(b^*) w_{t_j}^* s(\psi_2)a) \, \varphi_P(\sigma_{t_j}^{\varphi_P}(y^*)x) \\
&\to 0 \quad \text{as} \quad j \to \infty.
\end{align*}
This implies that $\psi_1 \otimes \varphi_P \not\sim_M \psi_2 \otimes \varphi_P$. This proves the claim and the proposition.
\end{proof}

Next, we prove that inclusions arising from free product von Neumann algebras are intertwining stable. Whenever $(N, \varphi_N)$ and $(P, \varphi_P)$ are von Neumann algebras endowed with faithful normal states, we denote by $(M, \varphi_M) = (N, \varphi_N) \ast (P, \varphi_P)$ the corresponding free product von Neumann algebra \cite{Vo85, Ue98}.

\begin{prop}\label{prop-stability-free}
Let $(M, \varphi_M) = (N, \varphi_N) \ast (P, \varphi_P)$ be any free product von Neumann algebra. Assume that $N_{\varphi_N}$ is a factor. Then $(M, N, \varphi)$ is intertwining stable.
\end{prop}

\begin{proof}
Denote by $\rE : M \to N$ the unique conditional expectation such that $\varphi_N \circ \rE = \varphi_M$. Let $\Theta \in \Aut(M)$ be any automorphism such that $\rE \circ \Theta = \Theta \circ \rE$ and $\Theta(\varphi) = \varphi$. Set $\Theta_N = \Theta|_N$. 

Let $\psi \in (N_\ast)^+$ and assume that $\Theta_N(\psi) \not\sim_N \psi$. We claim that $\psi \not\sim_N \varphi_N$. Indeed, otherwise we have $\psi \sim_N \varphi_N$ and so $\Theta_N(\psi) \sim_N \Theta_N(\varphi_N)$. Since $\Theta_N(\varphi_N) = \varphi_N$ and since $N_{\varphi_N}$ is a factor, Lemma \ref{lem-transitivity} implies that $\Theta_N(\psi) \sim_N \psi$, a contradiction. We have $\psi \not\sim_N \varphi_N$ and so $\Theta(\psi) \not\sim_N \varphi_N$. This further implies that $ \Theta_N(\psi) \oplus \Theta_N(\psi) \not\sim_{N \oplus N} \psi \oplus \varphi_N$. 

Choose $\Psi \in \mathcal E_{N, \psi}$.  For every $t \in \R$, set $w_t = [\rD\Psi:\rD \varphi_N]_t \in \mathcal U(N)$ and note that $ [\rD \psi :\rD \varphi_N]_t = s(\psi) [\rD \Psi : \rD \varphi_N]_t = s(\psi) w_t = w_t \sigma_t^{\varphi_N}(s(\psi)) \in s(\psi) N \sigma_t^{\varphi_N}(s(\psi))$. Using the chain rule (see \cite[Theorem VIII.3.7]{Ta03}) and since $\Theta_N(\varphi_N) = \varphi_N$, we have
\begin{align*}
\forall t \in \R, \quad [\rD \Theta_N(\Psi) :  \rD\Psi]_t &=  [\rD \Theta_N(\Psi) : \rD \Theta_N(\varphi_N)]_t  \, ([ \rD \Psi : \rD \varphi_N]_t)^*  \\
&= \Theta_N([\rD \Psi : \rD \varphi_N]_t) w_t^* \\
&= \Theta_N(w_t) w_t^*.
\end{align*}
Note that $ \Theta_N(w_t) w_t^* = [\rD \Theta(\Psi \circ \rE) :  \rD \, (\Psi \circ \rE)]_t $ for every $t \in \R$.

Since $ \Theta_N(\psi) \oplus \Theta_N(\psi) \not\sim_{N \oplus N} \psi \oplus \varphi_N$, Theorem \ref{thm-intertwining} implies that there exists a net $(t_j)_{j \in J}$ in $\R$ such that
\begin{align}\label{eq-convergence2}
\forall a, b \in N, \quad \lim_{j \to \infty} \varphi_N(\sigma_{t_j}^{\varphi_N}(b^*) \Theta_N(w_{t_j}^*s(\psi)) a) &= 0 \\  \label{eq-convergence3}
\forall a, b \in N, \quad \lim_{j \to \infty} \psi(\sigma_{t_j}^\Psi(b^*) w_{t_j}\Theta_N(w_{t_j}^* s(\psi))a)&= 0.
\end{align}

Set $P^\circ = \ker(\varphi) \cap P$ and $N^\circ = \ker(\varphi) \cap N$. For every $k \in \N$, denote by $\mathcal W_k$ the set of words of the form $w = a_1 \cdots a_{2k + 1}$, with alternating letters in $N$ and $P$, where $a_1, a_{2k + 1} \in N$ and $a_2 \in P^\circ, a_3 \in N^\circ, \dots, a_{2k - 1} \in N^\circ, a_{2k} \in P^\circ$. The linear span of the subsets $\mathcal W_k$ for $k \in \N$ is $\ast$-strongly dense in $M$. For every $k \geq 1$, we have $\rE(\mathcal W_k) = 0$.

Let $ k, \ell \in \N$. Let $a_1 \cdots a_{2k + 1} \in \mathcal W_k$ and $b_1 \cdots b_{2\ell + 1} \in \mathcal W_\ell$. Set
$$\alpha_j =\psi \left( \rE \left(\sigma_{t_j}^{\Psi \circ \rE}((b_1 \cdots b_{2\ell + 1})^* ) w_{t_j}\Theta_N(w_{t_j}^* s(\psi)) a_1 \cdots a_{2k + 1} \right) \right).$$
If $k = \ell = 0$, by \eqref{eq-convergence3}, we have that $ \alpha_j = \psi(\sigma_{t_j}^\Psi(b_1^*) w_{t_j}\Theta_N(w_{t_j}^* s(\psi))a_1) \to 0$ as $j \to \infty$. Next, we assume that $k + \ell \geq 1$. We have
\begin{align*}
\alpha_j &= \psi \left( \rE \left(\sigma_{t_j}^{\Psi \circ \rE}((b_1 \cdots b_{2\ell + 1})^* ) w_{t_j}\Theta_N(w_{t_j}^*s(\psi)) a_1 \cdots a_{2k + 1} \right) \right) \\
&= \psi \left( \rE \left(\sigma_{t_j}^{\Psi \circ \rE}(b_{2\ell + 1}^* \cdots b_{1}^*) w_{t_j}\Theta_N(w_{t_j}^*s(\psi)) a_1 \cdots a_{2k + 1} \right) \right) \\
&= \psi \left( \rE\left(w_{t_j} \sigma_{t_j}^{\varphi_M}(b_{2\ell + 1}^* \cdots b_{1}^*)  \Theta_N(w_{t_j}^*s(\psi)) a_1  \cdots a_{2k + 1} \right) \right) \\
&= \psi \left( \rE \left(w_{t_j} \sigma_{t_j}^{\varphi_N}(b_{2\ell + 1}^*) \cdots \sigma_{t_j}^{\varphi_P}( b_{2}^*) \; \sigma_{t_j}^{\varphi_N}(b_{1}^*) \Theta_N(w_{t_j}^*s(\psi)) a_1 \; a_2 \cdots a_{2k + 1} \right) \right).
\end{align*}
By \eqref{eq-convergence2}, we have $\lim_{j \to \infty} \varphi(\sigma_{t_j}^\varphi(b_{1}^*) \Theta(w_{t_j}^*s(\psi)) a_1) = 0$. Since $k + \ell \geq 1$, we have $\rE(\mathcal W_{k + \ell}) = 0$. Altogether, this implies that $\lim_{j \to \infty} \alpha_j = 0$. Since the linear span of the subsets $\mathcal W_k$ for $k \in \N$ is $\ast$-strongly dense in $M$, Theorem \ref{thm-intertwining} implies that $\Theta(\psi \circ \rE) \not\sim_M \psi \circ \rE$.
\end{proof}

\section{Proofs of the main theorems}

\subsection{Preliminaries on $R_\infty$}

	Denote by $R_\infty$ the Araki--Woods type ${\rm III_1}$ factor. Recall that $R_\infty$ is the unique amenable type ${\rm III}_1$ factor with separable predual \cite{Co75, Co85, Ha85}. Denote by $R$ the unique AFD type ${\rm II}_1$ factor. Then by uniqueness of the amenable type ${\rm III}_1$ factor, we have $R_\infty \cong R_\infty \ovt R$. Firstly, we recall the following description of automorphisms of $R_\infty$. 

\begin{lem}\label{lem-outer-conjugacy}
	Fix an isomorphism $R_\infty = R_\infty \ovt R$ and fix a faithful normal semifinite weight $\psi$ on $R_\infty$. Then for any $\Theta\in \Aut(R_\infty)$, there exist $u\in \mathcal U(R_\infty)$, $\beta\in \Aut(R)$, $\pi \in \Aut(R_\infty\ovt R)$, and $t\in \R$ such that 
	$$\Ad(u)\circ \Theta =\pi\circ( \sigma^\psi_t\otimes \beta)\circ \pi^{-1}.$$
\end{lem}
\begin{proof}
	Using \cite{KST89} and as explained in \cite[Section 3]{HS91} (before the proof of Theorem 3.1), up to outer conjugacy, any isomorphism on $R_\infty\ovt R$ is of the form $\sigma^\omega_t\otimes \beta$ for some $\beta\in \Aut(R)$ and $t\in \R$, where $\omega$ is a dominant weight on $R_\infty$. Since all modular automorphism groups are cocycle conjugate \cite{Co72}, we obtain the conclusion.
\end{proof}

Secondly, we provide many examples of almost periodic states on $R_\infty$.

\begin{lem}\label{lem-amenable-eigenvalue}
	The following assertions hold true. 
\begin{enumerate}[\rm (i)]
	\item For any countable dense subgroup $\Lambda < \R_{>0}$, there exists an extremal almost periodic faithful normal state  $\varphi_{\Lambda} \in (R_\infty)_\ast$ such that $\Lambda = \sigma_p(\Delta_{\varphi_{\Lambda}})$.

	\item Let $(P, \varphi_P)$ be any factor endowed with an almost periodic faithful normal state. Then for any countable dense subgroup $ \Lambda < \R_{>0}$ that contains $\sigma_p(\Delta_{\varphi_P})$, the state $\psi =\varphi_P\otimes \varphi_{\Lambda}$, where $\varphi_\Lambda$ is as in item $(\rm i)$, is an extremal almost periodic faithful normal state on $P \ovt R_\infty$ that satisfies $\Lambda = \sigma_p(\Delta_{\psi})$.
\end{enumerate}
\end{lem}
\begin{proof}
	$(\rm i)$ For every $\lambda \in (0, 1)$, denote by $(R_\lambda, \varphi_\lambda)$ the Powers type ${\rm III}_\lambda$ factor endowed with its canonical $\frac{2\pi}{|\log(\lambda)|}$-periodic faithful normal state. By uniqueness of the amenable type ${\rm III}_1$ factor, we may consider the isomorphism $R_\infty = \bigovt_{\lambda\in \Lambda \cap (0, 1)} (R_\lambda, \varphi_\lambda)$. We define $\varphi_\Lambda\in (R_\infty)_*$ as the tensor product state. Since $(R_\lambda)_{\varphi_\lambda}' \cap R_\lambda = \C 1$ for every $\lambda \in (0, 1)$, it follows that $\varphi_\Lambda$ is extremal. Moreover, we have $\sigma_p(\Delta_{\varphi_\Lambda}) = \Lambda$.
	
	$(\rm ii)$ It is easy to see that $\sigma_p(\Delta_{\psi}) = \sigma_p(\Delta_{\varphi_P}) \sigma_p(\Delta_{\varphi_\Lambda}) = \Lambda$. 
Next observe that
	$$ (P\ovt R_\infty)_{\psi}'\cap (P\ovt R_\infty) \subset (\C 1 \ovt (R_\infty)_{\varphi_\Lambda})'\cap (P\ovt R_\infty) = P\ovt \C1.$$
Take any $x\in (P\ovt R_\infty)_{\psi}'\cap (P\ovt R_\infty)$, which is of the form $x= y\otimes 1$. 
Take any $\lambda\in \sigma_p(\Delta_{\varphi_P})$ and any $\lambda$-eigenvector $p_\lambda\in P$. Since $\lambda^{-1}$ is an eigenvalue of $\Delta_{\varphi_\Lambda}$, there is a nonzero $\lambda^{-1}$-eigenvector $a_\lambda\in R_\infty$. Then $x$ commutes with $p_\lambda\otimes a_\lambda\in (P\ovt R_\infty)_{\psi}$, hence 
	$$p_\lambda y\otimes a_\lambda= (p_\lambda\otimes a_\lambda)x=x(p_\lambda\otimes a_\lambda)=yp_\lambda \otimes a_\lambda .$$
We obtain $p_\lambda y = yp_\lambda$. Since $P$ is generated by eigenvectors, this implies that $y\in \C 1$.
\end{proof}

\subsection{Proofs of Theorems \ref{thmA} and \ref{thmB}}

Theorems \ref{thmA} and \ref{thmB} will be deduced from the following general theorem.

\begin{thm}\label{thm-HS-conjecture}
	Let $(M, \varphi)$ be any type ${\rm III_1}$ factor endowed with an extremal almost periodic faithful normal state $\varphi$. Set $\Lambda = \sigma_p(\Delta_\varphi)$. Let $R_\infty \cong N \subset M$ be any Araki--Woods subfactor that is globally invariant under $\sigma^\varphi$ and such that $\varphi|_N$ is extremal and $\Lambda = \sigma_p(\Delta_{\varphi|_N})$. Denote by $\rE : M \to N$ the unique $\varphi$-preserving conditional expectation.

If $(M, N, \varphi)$ is intertwining stable, then $M$ satisfies {\bf HSC}.
\end{thm}

\begin{proof}
Denote by $G = \widehat{\Lambda}$ the dual group and by $\sigma^{\Lambda, \varphi} : G \curvearrowright M$ the unique continuous extension of $\sigma^\varphi : \R \curvearrowright M$. Let $\Theta\in \Aut(M)$ be any pointwise inner automorphism. Up to conjugating by a unitary, we may assume that $\Theta(\varphi)=\varphi$. Then Lemma \ref{lem-pointinner} implies that $\Theta$ is locally pointwise inner for $\varphi$ and Theorem \ref{thm-PI} implies that there exist $g\in G$ and $u\in \mathcal U(M_\varphi)$ such that $\Theta=\Ad(u)\circ \sigma_g^{\Lambda,\varphi}$. Then we may assume that $\Theta = \sigma_g^{\Lambda,\varphi}$ so that $\Theta(N)=N$ and $\Theta\circ \rE = \rE \circ \Theta$. Set $\Theta_N =\Theta|_N$, $\varphi_N = \varphi|_N$ and note that $\Theta_N=\sigma^{\Lambda,\varphi_N}_g$ since  $\sigma_p(\Delta_{\varphi_N})= \Lambda$.

Let $(R,\tau)$ be the unique AFD type $\rm II_1 $ factor endowed with its trace. Let $\psi \in (R_\infty)_\ast$ be any extremal almost periodic faithful normal state such that $\Lambda = \sigma_p(\Delta_\psi)$ (see Lemma \ref{lem-amenable-eigenvalue}). 
Fix an isomorphism $N = R_\infty \ovt R$. Then by Lemma \ref{lem-outer-conjugacy}, there exist $u\in \mathcal U(N)$, $\beta\in \Aut(R)$, $\pi \in \Aut(R_\infty\ovt R)$, and $t\in \R$ such that 
	$$\Theta_N^u =\Ad(u)\circ \Theta_{N} =\pi\circ( \sigma^\psi_t\otimes \beta)\circ \pi^{-1}.$$
Observe that $\sigma^\psi_t\otimes \beta$ preserves $\psi\otimes \tau$, hence $\Theta_N^u$ preserves $\pi(\psi\otimes \tau)= \omega$. 
Then observe that $\Theta^u = \Ad(u)\circ \Theta$ is pointwise inner on $M$. Since $(M, N, \varphi)$ is intertwining stable, Proposition \ref{prop-stable} implies that $\Theta^u_N$ is locally pointwise inner for $\omega$. Then by Theorem \ref{thm-PI}, $\Theta_N^u$ is inner on $N_\omega$. 

Here we claim that $\beta$ is inner. For this, observe that 
	$$N_\omega = (R_\infty\ovt R)_{\pi(\psi\otimes \tau)}=\pi((R_\infty\ovt R)_{\psi\otimes \tau} )=\pi((R_\infty)_{\psi}\ovt R ).$$
Since $\Theta_N^u$ is inner on $N_\omega$, we have that $\pi^{-1}\circ \Theta_N^u\circ \pi=\sigma_t^{\psi}\otimes \beta$ is inner on $(R_\infty)_{\psi}\ovt R$. We conclude that $\beta$ is inner on $R$. Thus, the claim is proven.

It follows that there exists $w\in \mathcal U(N)$ such that
	$$\Ad(w)\circ \Theta_N =\pi\circ( \sigma^\psi_t\otimes \id)\circ \pi^{-1} = \pi\circ( \sigma^{\psi\otimes \tau}_t)\circ \pi^{-1} = \sigma^{\pi(\psi\otimes \tau)}_t.$$
Since all modular automorphism groups are cocycle conjugate \cite{Co72}, up to modifying $w \in \mathcal U(N)$ if necessary, we obtain $\Ad(w)\circ \Theta_N = \sigma^{\varphi_N}_t$. 

We have proved that $\Ad(w)\circ \sigma^{\Lambda,\varphi_N}_g = \sigma^{\varphi_N}_t$. Since $N_{\varphi_N}$ is a factor and since $\sigma_p(\Delta_{\varphi_N})=\Lambda $, Lemma \ref{lem-injective-modular} implies that $g=t\in \R$. Thus, $\Theta = \Ad(w^*) \circ  \sigma_g^{\Lambda,\varphi} = \Ad(w^*) \circ \sigma_t^{\varphi}$. 
\end{proof}

\begin{proof}[Proof of Theorem \ref{thmA}]
	Fix an isomorphism $M = R_\infty \ovt P$. By Lemma \ref{lem-amenable-eigenvalue}, we may choose an extremal almost periodic faithful normal state $\varphi \in R_\infty$ and an almost periodic state $\varphi_P \in P_\ast$ so that $\psi = \varphi_{R_\infty} \otimes \varphi_P \in M_\ast$ is extremal and $\sigma_p(\Delta_\psi) = \sigma_p(\Delta_{\varphi_{R_\infty}})$. A combination of Proposition \ref{prop-stability-tensor} and Theorem \ref{thm-HS-conjecture} implies that $M$ satisfies {\bf HSC}.
\end{proof}

Next, we prove that many almost periodic free product factors satisfy {\bf HSC}.

\begin{thm}\label{thm-HS-free}
Write $N = R_\infty$. Let $(M, \varphi_M) = (N, \varphi_N) \ast (P, \varphi_P)$ be any free product von Neumann algebra where $\varphi_N \in N_\ast$ is an extremal almost periodic faithful normal state and $\varphi_P \in P_\ast$ is an almost periodic faithful normal state such that $\sigma_p(\Delta_{\varphi_P}) \subset \sigma_p(\Delta_{\varphi_N})$. Then the type ${\rm III}_1$ factor $M$ satisfies {\bf HSC}. 
\end{thm}

\begin{proof}
By \cite[Theorem 2.1]{Ue11}, $M$ is a full type ${\rm III}_1$ factor, $\varphi_M$ is an extremal almost periodic faithful normal state and $\sigma_p(\Delta_{\varphi_M}) = \sigma_p(\Delta_{\varphi_N})$. A combination of Proposition \ref{prop-stability-free} and Theorem \ref{thm-HS-conjecture} implies that $M$ satisfies {\bf HSC}.
\end{proof}

Theorem \ref{thmB} is simply a consequence of Theorem \ref{thm-HS-free} and \cite[Theorem B]{HN18}.

\bibliographystyle{plain}

\end{document}